\documentclass[11pt,leqno]{amsart}
\usepackage{amscd}
\usepackage{amssymb}
\usepackage{amsfonts}
\usepackage[centertags]{amsmath}
\usepackage{latexsym}
\usepackage{verbatim}
\usepackage{amsthm}
\usepackage{color}
\usepackage{pgfplots}
\usepackage{tikz}
\usetikzlibrary{intersections}
\usepackage{amsfonts}
\usepackage{amssymb}
\usepackage{amsmath}
\usepackage{amsthm}
\usepackage[english]{babel}
\usepackage{enumerate}
\usepackage{graphicx}
\usepackage{color}

\newcommand{\RR}{\mathbb{R}}

\newtheorem{theorem}{Theorem}[section]

\newtheorem{lemma}[theorem]{Lemma}

\theoremstyle{remark}

\DeclareMathOperator{\diam}{\mathrm{diam}}

\renewenvironment{proof}{
  \noindent{\it Proof.}\ }{\hspace*{\fill}
  \begin{math}\Box\end{math}\medskip}

\title{A remark on an overdetermined problem in Riemannian Geometry}

\author{Giulio Ciraolo, Luigi Vezzoni}

\date{\today}
\address{Dipartimento di Matematica e Informatica, Universit\`a di Palermo, Via Archirafi 34, 90123, Italy} \email{giulio.ciraolo@unipa.it}

\address{Dipartimento di Matematica G. Peano, Universit\`a di Torino, Via Carlo Alberto 10, 10123 Torino, Italy.} \email{luigi.vezzoni@unito.it}

\thanks{This work was partially supported by the project PRIN \lq\lq {\em Variet\`a reali e complesse: geometria, topologia e analisi armonica}\rq\rq ,  the projects FIRB \lq\lq {\em Differential Geometry and Geometric functions theory}\rq\rq and \lq\lq {\em Geometrical and Qualitative aspects of PDE}\rq\rq, and GNSAGA and GNAMPA (INdAM) of Italy.\\
}


\begin{document}

\maketitle

\begin{abstract}
Let $(M,g)$ be a Riemannian manifold with a distinguished point $O$ and assume that the geodesic distance $d$ from  $O$ is an isoparametric function. Let $\Omega\subset M$ be a bounded domain, with $O \in \Omega$, and consider the problem $\Delta_p u = -1$ in $\Omega$ with $u=0$ on $\partial \Omega$, where $\Delta_p$ is the $p$-Laplacian of $g$. We prove that if the normal derivative  $\partial_{\nu}u$ of $u$ along the boundary of $\Omega$ is a function of $d$ satisfying suitable conditions, then $\Omega$ must be a geodesic ball. In particular, our result applies to open balls of $\RR^n$ equipped with a rotationally symmetric metric of the form $g=dt^2+\rho^2(t)\,g_S$, where $g_S$ is the standard metric of the sphere. 
\end{abstract}

\section{Introduction}
In this note we consider an overdetermined problem in Riemannian Geometry. 
An overdetermined problem usually consists in a partial differential equation with \lq\lq too many\rq\rq prescribed boundary conditions.  Typically, these kinds of problems are not well-posed and the existence of a solution imposes strong restrictions on the shape of the domain where the problem is defined. Consequently, the research in overdetermined problems usually consists in classifying all the possible domains where the problem is well-posed.
A central result in this context was obtained by Serrin in his seminal paper \cite{Se}. The today known {\em Serrin's overdetermined problem} consists in the torsion problem 
\begin{equation} \label{pb_serrin1}
\begin{cases}
\Delta u = -1 & \textmd{in } \Omega \,, \\
u= 0 & \textmd{on } \partial \Omega \,, \
\end{cases}
\end{equation}
where $\Omega$ is a bounded domain in $\RR^n$,
together with a constant Neumann condition at the boundary $\partial \Omega$:
\begin{equation} \label{pb_serrin2}
u_\nu = const \quad \textmd{on } \partial \Omega \,.
\end{equation}
In \cite{Se} Serrin proved that \eqref{pb_serrin1}-\eqref{pb_serrin2} admits a solution if and only if $\Omega$ is a ball whose 
radius, in view of the divergence theorem, is determined by the constant in condition \eqref{pb_serrin2}; moreover the solution $u$ is radially symmetric. Other proofs and generalizations of Serrin's theorem can be found for instance in \cite{BNST,CH,We}. 

A related problem was considered by Greco in \cite{Gr}, where it is investigated the following question:

\medskip
{\em Given a point $O \in \RR^n$, $n\geq 2$, which overdetermined conditions force $\Omega$ to be a ball about $O$?}

\medskip 
\noindent The main result in \cite{Gr}, in its simplest form, is the following: 

\smallskip  
{\em Let $\Omega$ be a bounded domain in $\RR^n$ containing the fixed point $O$ and assume that there exists a solution $u$ of \eqref{pb_serrin1} in $\Omega$ satisfying $\partial_\nu u = c|x|$ at every $x \in \partial \Omega$ for some constant $c$. Then $\Omega$ is a ball centered at $O$.}

\smallskip 
The assumption $O\in \Omega$ in the statement above cannot be dropped in general, as pointed out in some examples in \cite{Gr}. For some related results we refer to \cite{BS,Gr2}. 

The problem posed by Greco in \cite{Gr} still makes sense in the Riemannian setting, where $\RR^n$ is replaced by a smooth $n$-dimensional manifold $M$, the Euclidean metric by a Riemannian metric $g$  on $M$ and Euclidean balls by geodesic balls about a fixed point $O$. 
In the present paper, we generalize some of the results in \cite{Gr} to the Riemannian setting by assuming that the distance function from the fixed point $O$ is isoparametric, i.e. it is of class $C^2$ in $M\setminus \{O\}$ and there exists a continuous function $\eta$ such that $\Delta d = \eta (d)$ in $M\setminus \{O\}$ (see \cite{Wang}). A related result can be found in \cite{Enciso_Peralta}. 

We shall use the following notation: $B_r$ denotes the geodesic ball of radius $r$ centered at $O$; $|B_r|$ and $|\partial B_r|$ are the volume and the perimeter of $B_r$, respectively,
\begin{equation} \label{Phi_def}
\Phi(r) := \left(\frac{|B_r|}{|\partial B_r|}\right)^{\frac{1}{p-1}} \,,
\end{equation}
and $\Delta_p$ is the $p$-Laplacian operator. 
Notice that $\Phi(r)$ is exactly the value of the interior normal derivative at the boundary of the solution of the problem
$$
\begin{cases}
\Delta_p v = -1 & \textmd{in } B_r \,,\\
v= 0 & \textmd{on } \partial B_r \,,
\end{cases}
$$
which is constant if one assumes that the distance function is isoparametric (see Lemma \ref{lemma_ball} below).

Our main result is the following. 

\begin{theorem} \label{thm_riemann}
Let $(M,g)$ be a Riemannian manifold and assume that the distance function $d$ from a fixed point $O$ is isoparametric. Let $\Omega \subset M$ be a bounded domain with $O \in \Omega$. Assume that there exists a solution $u$ to
\begin{equation} \label{pb_general}
\begin{cases}
\Delta_p u = -1 & \textmd{in } \Omega \,,\\
u= 0 & \textmd{on } \partial \Omega \,,\\
\partial_\nu u=f\circ d & \textmd{on } \partial \Omega\,,
\end{cases}
\end{equation}
with $p>1$, where $f$ satisfies one of the following two conditions:
\begin{itemize}
\item[(i)] the function $f(t)/\Phi(t)$ is monotone nondecreasing in $(0,\diam \Omega)$;

\item[(ii)] there exists $R>0$ such that $f(R)/\Phi(R)=1$, $f(r)/\Psi(r)>1$ for $r>R$, and $f(r)/\Phi(r)<1$ for $r<R$.
\end{itemize}
Then $\Omega$ is a geodesic ball centered at $O$. Moreover, in case (ii) we have that $\Omega=B_R$. 
\end{theorem}

Our proof relies on comparison principles. More specifically, our approach is based on the comparison of the solution in $\Omega$ with the solution in a ball about $O$. \\

As far as we know, few overdetermined problems have been studied in a Riemannian setting (see \cite{CV,Enciso_Peralta,FMV,FSV,FV}), where classical tools for proving symmetry like the method of moving planes can not be employed (at least in a standard way). Our approach is close to the one in \cite{Enciso_Peralta} where, by using comparison principles, the authors prove that if there exists a lower bounded nonconstant function $u$ which is p-harmonic ($1<p<n$) in a punctured domain and such that $u$ and $u_\nu$ are constant on $\partial \Omega$, then $u$ is radial and $\partial \Omega$ is a geodesic sphere.

\medskip
\noindent
 {\em{Acknowledgements}}. 
The second author is grateful to the organizers of \lq\lq {\em Geometric Properties for Parabolic and Elliptic PDE's 4th Italian-Japanese Workshop}" for the invitation and the very kind hospitality during the workshop.

\section{The Euclidean case} \label{section_toy}
In this preliminary section we consider the basic case when the Riemannian manifold is the Euclidean space, $O$ is the origin of $\RR^n$ and the $p$-Laplacian is the usual Laplacian operator (i.e. $p=2$). 

Let $\Omega$ be a bounded domain in the Euclidean space containing the origin $O$ and consider the overdetermined problem 
\begin{equation} \label{overd_pb_Euclidean}
\begin{cases}
\Delta u = -1 & \textmd{in } \Omega \,,\\
u= 0 & \textmd{on } \partial \Omega \,,\\
\partial_\nu u(x)=f(|x|) & \textmd{on } \partial \Omega\,,
\end{cases}
\end{equation}
where $\nu$ denotes the normal inward to $\partial \Omega$ and $f:(0,+\infty) \to (0,+\infty)$ is a continuous function satisfying certain conditions which we specify later. Since both Dirichlet and Neumann boundary conditions are imposed on $\partial \Omega$, problem \eqref{overd_pb_Euclidean} is not well-posed unless $f$ and $\Omega$ satisfy some compatibility conditions. Our goal is to consider conditions on $f$ which imply that $\Omega$ is a ball centered at the origin. 

The scheme that we have in mind is the following. Let $B_{r_0}$ be the largest ball contained in $\Omega$ centered at the origin and let $B_{r_1}$ be the smallest ball containing $\Omega$ centered at the origin. If we denote by $v^r$ the solution of 
$$
\begin{cases}
\Delta v = -1 & \textmd{in } B_r \,,\\
v= 0 & \textmd{on } \partial B_r \,,
\end{cases}
$$
then we have that $v^r$ is radially symmetric and $\partial_\nu v^r$ is constant on $\partial B_r$ and is given by
$$
\partial_\nu v^r= \frac{|B_r|}{|\partial B_r|} = \frac{r}{n} \quad \textmd{ on } \partial B_r \,.
$$
By comparison principle, we have that $v^{r_0} \leq u$ in $B_{r_0}$ and $v^{r_1} \geq u$ in $\Omega$ and, by looking at the normal derivatives at the tangency points between $\Omega$ and $B_{r_0}$ and $B_{r_1}$, we readily obtain that 
\begin{equation} \label{cond_on_f_Eucl}
\frac{r_0}{n} \leq f(r_0) \quad \textmd{ and } \quad f(r_1) \leq \frac{r_1}{n}  \,.
\end{equation}
Now, let 
$$
F(t)=\frac{n f(t)}{t} \,.
$$
We have the following results. The first one is essentially contained in \cite{Gr}.
\begin{itemize}
\item[\emph{Case 1.}] \emph{If $F(t)$ is monotone nondecreasing then $\Omega$ is a ball centered at the origin.}

\smallskip
\noindent Indeed, if $F(r_0)<F(r_1)$ then \eqref{cond_on_f_Eucl} immediately implies that $r_0=r_1$ and then $\Omega$ is a ball. If $F(r_0)=F(r_1)$, from  \eqref{cond_on_f_Eucl} we have that
$$
\partial_\nu u(x) = \frac{|x|}{n} 
$$
for every $x \in \partial \Omega$. Hence, the function $w=u-v^{r_0}$ satisfies 
$$
\begin{cases}
\Delta w = 0 & \textmd{in } B_{r_0} \,,\\
w \geq 0 & \textmd{in } \partial B_{r_0} \,, \\
\partial_\nu w(p) = 0 \,, &  
\end{cases}
$$
where $p\in \partial  \Omega$ is a tangency point between $B_{r_0}$ and $\Omega$. From Hopf's boundary point Lemma we obtain that $w \equiv 0$ in $\overline{B}_{r_0}$ and hence that $\Omega$ is a ball.

\vspace{0.5em}

\item[\emph{Case 2.}]   \emph{If there exists $R$ such that} 
\begin{equation} \label{overd_onodera} 
F(R)=1\,, \quad F(r)>1 \ \textmd{ if }\ r>R\,, \quad  F(R)<1 \ \textmd{ if } r < R \,,
\end{equation}
\emph{then $\Omega$ is the ball of radius $R$ centered at the origin.}\

\smallskip
\noindent This case is simpler then the previous one. Indeed, from \eqref{cond_on_f_Eucl} we have that $F(r_0) \geq 1$ and $F(r_1) \leq 1$, which imply that $r_0 \geq R$ and $r_1 \leq R$. Since $r_0 \leq r_1$, we conclude.
\end{itemize}

\section{Riemannian setting}
In this section we study the overdetermined problem \eqref{overd_pb_Euclidean} in a Riemannian manifold. 

Let $(M,g)$ be an dimensional Riemannian manifold with a fixed point $O$. We recall that given a $C^2$ map $u\colon M\to \RR$ its Laplacian is defined by 
$$
\Delta u={\rm div}(Du)\,,
$$
where $Du$ is the gradient of $u$ (which in the Riemannian setting is defined as the $g$-dual of the differential of $u$) and ${\rm div}$ is the divergence. 
$\Delta u$ writes in local coordinates $(x^1,\ldots,x^n)$ as 
$$
\Delta u = \frac{1}{\sqrt{|g|}} \partial_{x^j} \left( g^{jk} \sqrt{|g|} \partial_{x^k} u \right) \,.
$$
We further recall that the $p$-Laplacian operator  on a Riemannian manifold is  defined by 
$$
\Delta_pu = {\rm div}(|Du|^{p-2}Du), \quad p>1\,. 
$$
In this context a function $u\colon M\to \RR$ is called {\em radial} if can be written as $u=f\circ d$ for some real function $f$. Radial functions are usually studied in polar coordinates. Here we recall that if $0<\delta$ is less than the injectivity radius of $M$ at $O$, then $\exp_O$ induces polar coordinates $(r,q)\in (0,\delta)\times S^{n-1}$ on $B_\delta$ induced by the usual polar coordinates on the tangent space to $M$ at $O$. If $u=f\circ d$ is a radial function, then its Laplacian in polar coordinates takes the following expression 
$$
\Delta u=\partial^2_{rr} u+\left(\frac{n-1}{r}+\frac{\partial_r\det (d\exp_O)}{\det(d\exp_O)}\right)\partial_r u
$$ 
(see e.g. \cite{rosenberg}). In particular, if $u=d$ we get 
$$
\Delta d=\frac{n-1}{r}+\frac{\partial_r\det (d\exp_O)}{\det(d\exp_O)}\,.
$$
We notice that assuming $d$ isoparametric in the geodesic ball $B_{\delta}$ is equivalent to assume that the quantity
$$
\frac{\partial_r\det (d\exp_O)}{\det(d\exp_O)}
$$ 
defines a function of the distance itself. Moreover, we remark that 
\begin{equation} \label{dist_in_O}
\lim_{d\to 0} \Delta d = + \infty \,.
\end{equation}

We have the following lemma.

\begin{lemma} \label{lemma_ball}
Let $(M,g)$ be a Riemannian manifold and assume that the distance $d$ from a point $O\in M$ is isoparametric. Let $v^r$ be the solution of\begin{equation} \label{pLaplace_Br}
\begin{cases}
\Delta_p v^r = -1 & \textmd{in } B_r \,, \\
v^r=0 & \textmd{on } \partial B_r \,.
\end{cases}
\end{equation}
Then $v^r$ is a function which depends only on the distance from the origin and is given by
\begin{equation} \label{v^r_riemann}
v^r(x)= \int_{|x|}^r e^{\frac{1}{p-1} \int_t^r \eta(s) ds } \left[ \frac{|B_r|}{|\partial B_r|} - \int_t^r e^{-\int_\tau^r \eta(s) ds} d\tau  \right]^{\frac{1}{p-1}}  dt \,.
\end{equation}
In particular, $\partial_\nu v^r$ is constant on $\partial B_r$ and is given by
\begin{equation} \label{partial_nu_v^r}
\partial_\nu v^r = \left(\frac{|B_r|}{|\partial B_r|}\right)^{\frac{1}{p-1}}  \,.
\end{equation}
\end{lemma}

\begin{proof}
We look for a solution of the form $v^r(x)=V(d(x))$. Since $d$ is isoparametric, with $\Delta d = \eta(d)$, and the gradient of the distance function has unit norm in $M \setminus \{O\}$, then we have that $V$ satisfies 
$$
|V'(t)|^{p-2} \left[ (p-1) V''(t) + \eta(t) V'(t) \right] = -1,
$$
and hence
$$
\frac{d}{dt} |V'(t)|^{p-1} + \eta(t) |V'(t)|^{p-1} = 1.
$$
Therefore
$$
|V'(t)|^{p-1} = e^{\int_t^r \eta(s) ds } \left[ |V'(r)|^{p-1} - \int_t^r e^{-\int_\tau^r \eta(s) ds} d\tau  \right] \,.
$$
Since we are looking for a solution in $B_r$ which depends only on $d$, from \eqref{pLaplace_Br} and the divergence theorem we have that $|V'(r)|^{p-1}=|B_r|/|\partial B_r|$ and hence   
$$
V'(t) = - e^{\frac{1}{p-1} \int_t^r \eta(s) ds } \left[ \frac{|B_r|}{|\partial B_r|} - \int_t^r e^{-\int_\tau^r \eta(s) ds} d\tau  \right]^{\frac{1}{p-1}} \,,
$$
and the expression for $v^r$ follows. From \eqref{dist_in_O} we have that $V'(0)=0$ and then $v^r \in C^{1,\alpha}$ in $B_r$ and satisfies \eqref{pLaplace_Br}. 
\end{proof}

\noindent We are ready to prove Theorem \ref{thm_riemann}.

\medskip

\emph{Proof of Theorem $\ref{thm_riemann}.$}  
We firstly give some remarks on the regularity of the solution. From elliptic regularity theory we have that $u \in C^{1,\alpha}(\Omega)$ (see \cite{DiBen,Lew,Tol}) and $u \in C^{2,\alpha}$ in a neighborhood of any point where $|\nabla u| \neq 0$ (see \cite{GT}). About the regularity at the boundary, we notice that we have by assumption that $|\nabla u| \neq 0$ on $\partial \Omega$ and hence  $|\nabla u| \neq 0$ in a tubular neighborhood of $\partial \Omega$. From \cite{GL,Vo} we obtain that $\partial \Omega$ is of class $C^2$ and from \cite{Li} we have that $u \in C^{1,\alpha}( \overline{\Omega})$.
 
Now, we observe that $u>0$ in $\Omega$. Indeed, the boundary condition in \eqref{pb_general} implies that $u>0$ in a neighborhood of $\partial \Omega$. If $u=0$ at some interior point of $\Omega$, then the strong maximum principle (see \cite{PS}) implies that $u \equiv 0$ in $\Omega$, which gives a contradiction. Hence, $u>0$ in $\Omega$.

We define $r_0$ and $r_1$ as follows
$$
r_0=\sup\{r>0 :\ B_r \subset \Omega \} \quad \textmd{ and } \quad r_1=\inf \{ r>0:\ \Omega \subset B_r \} \,,
$$
and we denote by $x_i$ a tangency points between $\partial B_{r_i}$ and $\partial \Omega$, for $i=0,1$.
  
As in the Euclidean case in section \ref{section_toy}, the proof is based on the comparison between $u$ and the solutions of the $p$-torsion problem in $B_{r_0}$ and $B_{r_1}$. Since $B_{r_0} \subseteq \Omega \subseteq B_{r_1}$, by the weak comparison principle (see \cite{HKM,PS}) we have that $v^{r_0} \leq u$ in $B_{r_0}$ and $u \leq v^{r_1}$ in $\Omega$, where $v^{r_0}$ and $v^{r_1}$ are given by \eqref{v^r_riemann}. 

Since $x_i$ is a tangency point between $\partial B_{r_i}$ and $\partial \Omega$, the inward normal vectors to $\partial B_{r_i}$ and to $\partial \Omega$ at $x_i$ agree and $d(x_i)=r_i$ for $i=0,1$. Moreover, $v^{r_i}(x_i) = u(x_i) = 0 $, and by comparison we have that 
$$
\Phi(r_0) = \partial_{\nu} v^{r_0}(x_0) \leq \partial_{\nu} u(x_0) \quad \textmd{ and } \quad \partial_{\nu} u(x_1) \leq \partial_{\nu} v^{r_1}(x_1) = \Phi(r_1)\,,
$$
and hence
\begin{equation} \label{cond_riem}
1 \leq \frac{f(r_0)}{\Phi(r_0)} \quad \textmd{ and } \quad  \frac{f(r_1)}{\Phi(r_1)} \leq 1 \,.
\end{equation}
If we assume that case (ii) in the assertion of the theorem occurs, then \eqref{cond_riem} implies that $r_1 \leq R \leq r_0$, and hence $r_0=r_1=R$.

In case (i), we have that \eqref{cond_riem} implies  
$$
\frac{f(r)}{\Phi(r)} = 1 \quad \textmd{for every } r_0 \leq r \leq r_1 \,,
$$
and hence  
$$
\partial_\nu u(x) = \Phi(|x|)  \quad \textmd{for every } x \in \partial \Omega \,.
$$
In particular, we have that 
\begin{equation} \label{ulivi}
\partial_\nu u(x_0) = \partial_\nu v^{r_0}(x_0) \,.
\end{equation} 
Since $\partial_\nu u(x_0)>0$, there exists $\rho>0$ such that $|\nabla u|  \neq 0$ in $B_{\rho}(x_0) \cap \Omega$. By choosing $\rho < r_0$ we also have that $|\nabla v^{r_0}|  \neq 0$ in $W:=B_{\rho}(x_0) \cap B_{r_0}$. By standard elliptic regularity theory, we have that $u$ and $v^{r_0}$ are classical solutions of $\Delta_p u = -1$ in $W$ and the difference $u-v^{r_0}$ is nonnegative and satisfies a linear uniformly elliptic equation in $W$:
$$
\begin{cases}
L(u-v^{r_0})=0 \quad \textmd{and} \quad u-v^{r_0} \geq 0  & \textmd{in } W \,, \\
\partial_{\nu} (u-v^{r_0} )(x_0) = 0 \,.
\end{cases}
$$
By Hopf's Lemma (see \cite{Hopf}) we have that $u=v^{r_0}$ in $W$. In particular, we obtain that $u=0$ in $\partial B_\rho (x_0) \cap B_{r_0}$, which implies that $\partial B_{r_0}$ $\partial \Omega$ coincide in a open neighborhood of $x_0$. More precisely, we have proved that the set of tangency points between $\partial \Omega$ and $\partial B_{r_0}$ is both open and closed, and hence we have that $\partial \Omega = \partial B_{r_0}$, i.e. $\Omega$ is a ball.

%
%

\section{Examples}

Theorem \ref{thm_riemann} can be applied to open balls in $\mathbb{R}^n$ equipped with a rotationally symmetric metric. More precisely, let $\bar r\in \mathbb{R}\cup \infty$ be fixed  and consider the open ball $B_{\bar r}$ centered at the origin $O$ of $\RR^n$ of radius $\bar r$ equipped with a Riemannian metric $g$ which in polar coordinates reads as 
$$
g=dt^2+\rho^2g_{S}\,,
$$
where $\rho\colon [0,\bar r)\to \mathbb R $ is as smooth function such that 
$$
\rho(0)=0\,,\quad \rho(t)>0 \,,
$$
for every $t\in [0,\bar r)$ and $g_S$ is the standard metric on the unitary $(n-1)$--dimensional sphere $S^{n-1}$. In this setting the geodesic distance $d$ of a generic point $p\in B_{\bar r}$ from $O$ is given by the Euclidean norm of $p$, since $t\mapsto t p$ is a minimal geodesic connecting the origin to the point $p$ for $t\in [0,\bar r)$. Moreover, if $u\colon B_{\bar r}\to \mathbb R $ is a smooth radial function, then its Laplacian with respect to $g$ takes the following expression
$$
\Delta u = \partial_{tt}^2 u + (n-1) \frac{\rho'}{\rho} \partial_t u\,,
$$
and consequently 
$$
\Delta d (x) = (n-1) \frac{\rho'(d (x))}{\rho(d (x))}\, =: \eta (d (x)) \,,
$$
which shows that $d$ is isoparametric.  Notice that in this setting the statement of theorem \ref{thm_riemann} implies that $\Omega$ is an Euclidean ball, since geodesic balls centered at $O$ are exactly the Euclidean balls. 

Rotationally symmetric spaces include space form models as particular cases: the Euclidean space, the Hyperbolic space and the unitary sphere, where the function $\rho$ takes the following expression: 
\begin{itemize}
\item $\rho(t)=t$ in the Euclidean case;
\item $\rho(t)=\sinh t$ in the Hyperbolic case;
\item $\rho(t)=\sin t$ in the spheric case. 
\end{itemize}
Note that the map $v^r$ in lemma \ref{lemma_ball} in the Euclidean case takes the following expression   
$$
 v^r(x)=\left(\frac{p-1}{p}\right) \frac{r^{\frac{p}{p-1}}-|x|^{\frac{p}{p-1}}}{n^{\frac{1}{p-1}}} \,.
$$

%
%
%

\end{document}